\documentclass[12pt]{amsart}
\usepackage{latexsym}
\usepackage{graphicx, color, epsfig, verbatim, subcaption, epsf, enumitem} 
\usepackage{amssymb, amsmath,mathrsfs}
\usepackage{graphicx}
\usepackage{tikz}
\usepackage{pgfplots}
\usepgfplotslibrary{fillbetween}
\newcommand\beq{\begin{equation}}

\newcommand\eeq{\end{equation}}

\usepackage{float}

\usepackage{setspace}
\makeatletter
\@namedef{subjclassname@2020}{%
  \textup{2020} Mathematics Subject Classification}

\usepackage[pagebackref,hypertexnames=false, colorlinks, citecolor=red, linkcolor=red]{hyperref}

\newcommand{\McC}{\raise.5ex\hbox{c}}

\newcommand\bbm{\begin{bmatrix}}
\newcommand\ebm{\end{bmatrix}}
\newcommand\dd{\mathrm d}

\newtheorem{theorem}{Theorem}[section]

\newtheorem{lemma}[theorem]{Lemma}
\newtheorem{conjecture}[theorem]{Conjecture}

\newtheorem*{theorem*}{Theorem}
\newtheorem*{conjecture*}{Conjecture}
\newtheorem{corollary}[theorem]{Corollary}
\newtheorem*{corollary*}{Corollary}

\newtheorem*{proposition*}{Proposition}

\def\bb{\begin{color}{blue}}
\def\bg{\begin{color}{green}}
\def\br{\begin{color}{red}}
\def\eg{\end{color}}
\def\er{\end{color}}
\def\eb{\end{color}}

\theoremstyle{remark}
\newtheorem{remark}[theorem]{Remark}

\author[Bickel]{Kelly Bickel$^\dagger$}
\address{Department of Mathematics, Bucknell University, 360 Olin Science Building, Lewisburg, PA 17837, USA.}
\email{kelly.bickel@bucknell.edu}
\thanks{$\dagger$ Research supported in part by National Science Foundation DMS grant \#2000088.}

\author[Pascoe]{J. E. Pascoe$^\ddagger$}
\address{Department of Mathematics, University of Florida, 1400 Stadium Rd, Gainesville, FL 32611, USA.}
\email{pascoej@ufl.edu}
\thanks{$\ddagger$ Research supported by National Science Foundation DMS grant \#1953963.}

\author[Sargent]{Meredith Sargent$^*$}
\address{University of Manitoba, Department of Mathematics, 420 Machray Hall, 186 Dysart Road University of Manitoba, Winnipeg, MB R3T 2N2, Canada.}
\email{meredithsargent@gmail.com}
\thanks{$*$ Research conducted for this paper is supported by the Pacific Institute for the Mathematical Sciences (PIMS). The research and findings may not reflect those of the Institute}

\begin{document}
\title[Zero-free regions]{Zero-free regions near a line}
\date{\today}
\begin{abstract}
 We analyze metrics for how close an entire function of genus one is to being real rooted. These metrics arise from truncated Hankel matrix positivity-type conditions built from power series coefficients at each real point.  Specifically, if such a function satisfies our positivity conditions and has well-spaced zeros,  we show that all of its zeros have to (in some explicitly quantified sense) be far away from the real axis. 

The obvious interesting example arises from the Riemann zeta function, where our positivity conditions yield a family of relaxations of the Riemann hypothesis. One might guess that as we tighten our relaxation, the zeros of the zeta function must be close to the critical line. We show that the opposite occurs:  any potential complex zeros
are forced to be farther and farther away from the critical line.

\end{abstract}
\maketitle


\section{Introduction}
\subsection{Motivation and setup} An entire function of genus one is a function of the form
		\beq \label{eqn:entiref}
			f(z) = z^\ell e^{d_1+d_2z}\prod_{i=1}^{\infty} \left( 1 - \frac{z}{\lambda_i} \right) e^{z/\lambda_i},
		\eeq
	where $\ell$ is a nonnegative integer and $\sum_i \tfrac{1}{|\lambda_i|^2} <\infty$. If both $d_2$ is real and all of the roots $\lambda_i$ are real, then $f$ is in the Laguerre-P\'olya class, the class of entire functions that are locally, uniformly limits of sequences of polynomials with real zeros  \cite{debrangesentire}.  As an aside, every function $h$ in the Laguerre-P\'olya class actually satisfies $h(z) = e^{dz^2} f(z)$, where $f$ is of form \eqref{eqn:entiref} with $d_2$ and the $\lambda_i$ real and  $d \le 0$.
 The goal of this investigation, broadly speaking, is to detect how close an entire function of genus one is to being real rooted by developing a hierarchy of relaxations of the Laguerre-P\'olya class (when $d=0$).  

To motivate our Laguerre-P\'olya relaxations, first note that given $f$ in \eqref{eqn:entiref}, its negative log derivative has  formula 
		$$g(z):= -\frac{d}{dz}\log f(z) =  \frac{-\ell}{z} - d_2 +
		\sum_{i=1}^{\infty} \frac{z}{ \lambda_i(\lambda_i - z)}.$$
Elementary calculations shows that $g(z) + d_2$ maps the upper half plane to itself and the lower half plane to itself if and only if all of the $\lambda_i$ are real.
	Expanding the Laurent series for the logarithmic derivative at $0$ gives
		$$g(z) = \frac{-\ell}{z} - d_2 + \sum^{\infty}_{n=1} a_n z^n,$$
	where $a_n = \sum \frac{1}{\lambda_i^{n+1}}.$
	To connect this to a matrix condition,  define the measure $\mu = \sum \frac{1}{\lambda_i^2} \delta_{1/\lambda_i}$ on $\mathbb{C}$ and
	note that the measure $\mu$ is positive if and only if all of the $\lambda_i$ are real. Moreover, 
	its moments satisfy
		$$a_n = \int w^{n-1} \dd \mu (w).$$
Write the infinite Hankel matrix
		\beq \label{hankelone}
			A=\bbm 
			a_{1} & a_2 & a_3 & \ldots \\
			a_{2} & a_3 & a_4 & \ldots\\
			a_{3} & a_4 & a_5 & \ldots\\
			\vdots & \vdots & \vdots & \ddots
			\ebm_{i,j }
			= \bbm \int w^{i+j-2}\dd \mu(w) \ebm_{i,j}.
		\eeq
	Nevanlinna's solution to the Hamburger moment problem \cite{nev22} implies
	that the infinite matrix $A$ in \eqref{hankelone} is positive semidefinite
	if and only if $\mu$ is positive if and only if all of the $\lambda_i$ are real.
	
	For each $x \in  \mathbb{R} \setminus \{\lambda_i\}$, one can similarly expand $g(z+x) = \sum^{\infty}_{n=0} a_n(x)z^n$ with $a_n(x) =  \sum \frac{1}{(\lambda_i-x)^{n+1}}$ for $n>0$
	and use analogues of the previous arguments to deduce that all of the $\lambda_i$ are real if and only if the infinite matrix 	$A(x)$ with entries $A(x)_{i,j} = a_{i+j-1}(x)$ 
	is positive semidefinite. Thus, $f$ is in the Laguerre-P\'olya class if and only if $d_2\in \mathbb{R}$ and any (or 
	every) such $A(x)$ is positive semidefinite.
	
	Our relaxations of the Laguerre-P\'olya class involve truncations
	of \eqref{hankelone} and the more general $A(x)$, as in the classical results of Dobsch-Donoghue \cite{Donoghue,Dobsch, 	bhatia} on L\"owner's theorem and $N$-matrix monotonicity. Specifically, for each $N \in \mathbb{N}$,
	we say $f$ is in the  \textbf{N-th order Laguerre-P\'olya class} (denoted \textbf{$\mathbf{N}$-LP}) if $d_2 \in \mathbb{R}$ and the matrix inequality
	 \begin{equation} \label{eqn:Nmatrix} A_N(x)= \begin{bmatrix} a_1(x) & a_2(x) & \dots & a_{N}(x)  \\
 a_2(x) & a_3(x) & & \\
 \vdots & &\ddots &  \vdots\\
 a_{N}(x) & &\dots & a_{2N-1}(x) 
  \end{bmatrix} \ge 0\end{equation}
holds for all $x \in \mathbb{R} \setminus  \{ \lambda_i\}$. Here, the notation $A_N(x) \ge 0$ means $A_N(x)$ is a positive semidefinite matrix.  Also, it is worth noting that these $N$-LP classes are nested and if $f$ is in $1$-LP, then any non-real zeros of $f$ must come in complex-conjugate pairs (so that $f$ is real on the real line), see Lemma \ref{lem:1O} for details.

As an aside, it is worth noting that all functions whose negative logarithmic derivative is a self-map of the upper half plane have a continuous version of the Hadamard factorization \cite{PascoeTrace}. Moreover, such functions satisfy the determinantal isoperimetric inequality
		\beq \label{detiso}
			\det f(A)f(C) \leq \det f(B)f(D)
		\eeq
	whenever $A \leq B \leq C$ and $D = A+B -C$, where $A, B, C$ are self-adjoint matrices of the same arbitrary size with spectrum contained in some interval in $\mathbb{R}$ where $f$ does not vanish
	and $f$ is evaluated on matrices via the functional calculus. While we do not belabor to prove the point as it is irrelevant to our current aims,  readers of culture will see that functions in the $N$-th order Laguerre-P\'olya class preserve the inequality \eqref{detiso} on $N$ by $N$ matrices via the classical
	Dobsch-Donaghue theorem \cite{Donoghue} combined with \cite[Theorem 3.3]{PascoeTrace}.

	We examine entire functions, and particularly those in the $N$-th order Laguerre-P\'olya class,
	whose zeros $\{\lambda_i\}$ satisfy reasonable spacing conditions.  To that end, 
	define the \textbf{spacing constant} $c$ of a Hadamard product $f$ as in \eqref{eqn:entiref} as $c=0$ if $f$ has a repeated zero and 
	\begin{equation} \label{eqn:c} c=\inf_{i \ne j}  \left\{ \left| \Re(\lambda_i-\lambda_j) \right|:  \lambda_i \ne \bar{\lambda}_j \right\},\end{equation}
	if $f$ has only simple zeros. 
	If the spacing constant is nonzero, then we call the function \textbf{spaced}.
	Define the \textbf{height} of a Hadamard product to be
		\begin{equation} \label{eqn:b}  b= \inf_i \left \{| \Im (\lambda_i)| : \Im (\lambda_i) \neq 0\right \}.\end{equation}
	If the infimum is taken over an empty set, we say the height is infinite.
	We define the \textbf{aperture} of a Hadamard product to be $\kappa = b/c$ if $c \neq 0.$ Otherwise
	we define $\kappa = \infty.$ If our function is spaced, then $\kappa = \infty$ implies the zeros are real.
	
	\subsection{Main Results} Our main results are of the following flavor: we assume a function is in the $N$-th order Laguerre-P\'olya class and then conclude that $\kappa$ has to be big where ``how big'' goes to infinity with $N.$
	Note that if $\kappa$ is large, there can be no non-real zeros near the real axis.
	For example, we obtain the following for functions in the first order Laguerre-P\'olya class.
	\begin{theorem} \label{thm:1}
		Let $f$ be an entire function of genus one such that $f(0)\neq 0.$
		If $f$ is in the first order Laguerre-P\'olya class,
		then
			$$\kappa \geq \frac{\sqrt{3}}{\pi}.$$
	\end{theorem}
This appears in Section \ref{sec:spaced} as Theorem \ref{lem:kappa}. More generally, we obtain the following, which appears later as  Theorem \ref{thm:NLP}:
	\begin{theorem} \label{thm:2}
		Let $f$ be an entire function of genus one such that $f(0)\neq 0.$
		If $f$ is in the N-th order Laguerre-P\'olya class then
		$$N \leq \frac{\ln\left( \frac{4\pi^2}{3}+4 \right)}{\ln 2} + \frac{\pi^3 \sqrt{3}}	{\ln 2} \left(\kappa + \sqrt{1+\kappa^2} \right)^6\left(1+\kappa^2\right)^{3/2} \approx \kappa^9. $$
	\end{theorem}
	We conjecture that the $9$-th order behavior is optimal.
	
	To prove our results, we use the following idea, which is similar to one arising in the $N$-matrix monotonicity literature \cite{Donoghue, heinevaara1, heinevaara2}.
	We test the Hankel matrix $A_N(x)$ on a vector 
		$$
		\bbm 
		q_{1} \\
		q_{2}\\
		\vdots \\
		q_N
		\ebm^*
		A_N(x)
		\bbm 
		q_{1} \\
		q_{2}\\
		\vdots\\
		q_N
		\ebm
		= \sum_i q(1/(\lambda_i-x))^2,$$
	where $q(z) = \sum^{N}_{j=1} q_jz^j,$ $q_j \in \mathbb{R}.$ 
	Then to a detect non-real zero $\lambda_k$, we try to simultaneously make both $q(1/(\lambda_k-x))^2$ big and negative 	and $q(1/(\lambda_i-x))^2$ small for $\lambda_i \in \mathbb{R}$. 
	If one wants a $q$ that is small on $\mathbb{R}$ and big off the real line, one might guess
	that the optimal things to try are Taylor polynomials for $\sin (rz).$
	Specifically, our crucial observation is a ``sine recovery lemma," which is perhaps a remarkable result on its own. (Indeed, we have now remarked about it in the introduction.)
	\begin{lemma}[Sine recovery lemma] 
		Let $f$ be an entire function of genus one such that $f(0)\neq 0.$
		If
		$\sum_i \sin(t/(\lambda_i-x))^2 \geq 0$ for all $t \in \mathbb{R}$ and $x\in \mathbb{R} \setminus \{\lambda_i\}$, then
		$\kappa = \infty.$ 
	\end{lemma}
	This appears at the end of Section \ref{sec:spaced} as Lemma \ref{lem:sine}.

We also briefly consider functions that satisfy a stronger spacing condition. We call these functions \textbf{strongly spaced}, which means that there are constants $\gamma, C>0$ such that if the zeros $\lambda_j$ are ordered with increasing, positive real part, then
\[\Re(\lambda_{j+1}-\lambda_j) \geq C| \Re(\lambda_j)|^{\gamma}\]
for all $j$ where $\lambda_{j+1} \ne \bar{\lambda}_j$. Under this assumption, we prove:
\begin{theorem} \label{thm:3} Let $f$ be  in the first order Laguerre-P\'olya class with $f(0) \ne 0$ and strongly spaced constants $C, \gamma>0$. If $\lambda_j$ is a zero of $f$ with $\Im(\lambda_j) \ne 0$, $\lambda_{j+1} = \bar{\lambda}_j$, and $j >1$, then 
\[ \left | \Im(\lambda_j) \right| \ge C\tfrac{\sqrt{3}}{2\pi} \left|\Re(\lambda_{\lfloor j/2 \rfloor} )\right|^{\gamma},\]
where $\lfloor j/2 \rfloor$ denotes the greatest integer less than or equal to $j/2$. 
\end{theorem}

This appears in Section \ref{StrongSpace} as Theorem \ref{thm:ss} (which also mentions the $j=1$ case). As the strongly spaced setting was outside of the bounds of our original investigation, we only analyze the behavior of strongly spaced functions in the  first order Laguerre-P\'olya class. However, we encourage the ambitious reader to explore  strongly spaced functions in the more general $N$-th order Laguerre-P\'olya class. 

For the remainder, we connect our results to the Riemann hypothesis and then prove our main results. 
 In Section \ref{sec:rh}, we use our Laguerre-P\'olya relaxations to obtain a family of relaxations of the Riemann hypothesis and assuming certain spacing conditions, use our main results to conjecture certain zero-free regions of the zeta function. 
Regarding spacing, we propose a related spacing constant in Section \ref{ss:rh} and discuss two alternate nonconstant spacing regimes and their applications in Section \ref{subsec:ssc}. In our final three sections, we establish preliminary facts about Hankel matrices associated to Hadamard products and then prove our main (previously-stated) results.

\section*{Acknowledgements}  
We thank David Farmer for significant concrete discussions and useful insights on preliminary drafts, including the proof of Remark \ref{rem:improve} on Theorem \ref{thm:1}.

\section{Connections to the Riemann hypothesis} \label{sec:rh}

In this section, we further motivate our Laguerre-P\'olya relaxations by considering their implications for the Riemann hypothesis. First, let the sequence $\{\rho_k\}$ denote the zeros of the Riemann zeta function with positive imaginary part, listed according to multiplicity. To connect our setting to the Riemann hypothesis, we need to change variables slightly. To that end, for each $k$,
define  $\eta_k \in \mathbb{C}$ by $\eta_k = \Im(\rho_k) +i ( \tfrac{1}{2} -\Re(\rho_k))$. Then
\begin{equation} \label{eqn:rhok} \rho_k = \Re(\rho_k) + i \Im(\rho_k) =\tfrac{1}{2} + i\left( \Im(\rho_k) +i ( \tfrac{1}{2} -\Re(\rho_k))\right)= \tfrac{1}{2} + i \eta_k.\end{equation}
Set $\lambda_k = \eta_k^2$ and let $\Lambda$ be the complex analytic function
$$\Lambda(z) = \prod_{k} \left(1-\frac{z}{\lambda_k}\right) = e^{-(\sum_k 1/\lambda_k)z} \prod_{k} \left(1-\frac{z}{\lambda_k}\right) e^{z/\lambda_k}.$$
The function $\Lambda$ satisfies $\Lambda(z) = \Xi(\sqrt{z})$, where $\Xi$ is a standard function defined from the Riemann zeta function by factoring out its trivial zeros and pole. The functional equation of the zeta
function is equivalent to the assertion that $\Xi$ is real on the real
line. 
It follows that the Riemann hypothesis holds if and only if each $\eta_k$ (or equivalently, each $\lambda_k$) is on the real line.

\subsection{Relaxation hierarchy for the Riemann hypothesis} \label{ss:rh}

The key observation motivating these investigations is the following: $\Lambda$ is in the $N$-th order Laguerre-P\'olya class for all $N$ if and only if the Riemann hypothesis holds, see \cite{PascoeTrace}. Based on that, we propose the following set of relaxations of the Riemann hypothesis for each $N\in \mathbb{N}.$
\begin{conjecture}[$N$-th Hankel relaxation Riemann hypothesis, $N$-HRRH]
	The function $\Lambda$ is  in the $N$-th order Laguerre-P\'olya class.
\end{conjecture}

 This condition is equivalent to saying that the matrix $A_N(x)$ in \eqref{eqn:Nmatrix} is positive semi-definite for $x \in \mathbb{R} \setminus \{\lambda_i\}$. Here, the $jk^{th}$ entry of $A_N(x)$ is
\[ a_{j+k-1}(x)= \sum_i \frac{1}{(\lambda_i-x)^{j+k}},\]
see Section \ref{sec:prelim}. It follows immediately that  each $A_N(x)$ is self-adjoint, 
since properties of the Riemann zeta function imply that the zeros of $\Lambda$ are real or  occur in complex conjugate pairs.  Then, checking the positivity of $A_N(x)$ is equivalent to showing that the principle minors of $A_N(x)$ (the determinants of the square submatrices obtained by removing a (possibly empty) set of rows and their corresponding columns from $A_N(x)$) are positive or zero. Because the entries of $A_N(x)$ come from the log derivatives of $\Lambda$ to order $2N$, these determinantal conditions are equivalent to a set of inequalities involving the derivatives of $\Lambda$ to order $2N.$

For example, if $N=1$, this conjecture says that for $x \in \mathbb{R} \setminus \{\lambda_i\}$, 
\[ 0 \le a_1(x) := -\frac{d^2}{dz^2}\left(  \log \Lambda(z) \right) 
\Big|_{z=x} = \frac{ \Lambda'(x)^2 - \Lambda''(x) \Lambda(x)}{\Lambda(x)^2}.\]

By continuity,  the numerator is nonnegative for every $x \in \mathbb{R}$, which implies the following simple form for $1$-HRRH:

\begin{conjecture}[$1$-HRRH]
$\Lambda'(x)^2-\Lambda''(x)\Lambda(x)\geq 0$ for real $x.$
\end{conjecture}

If $N=2$, our conjecture says that  $a_1(x)\ge0,$  $a_3(x)\ge 0,$ and $a_1(x)a_3(x) - a_2(x)^2 \ge 0$. One can rewrite these inequalities in terms of derivatives of $\Lambda$, but the result is not particularly illuminating. 
Instead, for $N \ge 2$, we suggest that the reader think about  our order-$N$ relaxations in terms of the matrix conditions or the polynomial criterion given in Lemma \ref{lem:poly}.

We also note that $1$-HRRH was originally conjectured by Csordas in \cite{CsordasConjecture}. Moreover, Csordas made some some higher order Tur\'an inequalities conjectures which together would prove the Riemann hypothesis,
although it not clear to us if they give rise to the exact same family of relaxations as ours, see \cite{Csordas, Csordas2, Dimitrov}.   Our relaxations may also be  connected to the following collection of conjectures studied by Farmer in \cite{Farmer2}: for $\Lambda$ and each of its derivatives, all local maxima are positive and all local minima are negative.  Again, these conjectures encompass $1$-HRRH, but it is not clear to us how the higher-order conjectures interact with our order-$N$ relaxations.

We can apply our main results (Theorems \ref{thm:1}, \ref{thm:2}, \ref{thm:3}) to $\Lambda$ if $\Lambda$ is in $N$-LP (i.e.~if $N$-HRRH holds) and if $\Lambda$ is {spaced} (or {strongly spaced}).  Under those assumptions, our main results imply that for each $N$, $\Lambda$ possesses a specific zero-free region near the real line, and as $N$ goes to infinity, $\Lambda$ has no zeros off of the real line. Of course, the specific zero-free regions will depend on the spacing constant for $\Lambda$, as defined in \eqref{eqn:c}.  Unfortunately, explicit bounds on the spacing of the zeta function's zeros are generally not known and it is an open question as to whether the zeros are even all simple, see \cite{Trudgian} and the references therein for recent results. Still, using the numerical evidence in the Odlyzko tables \cite{OdlyzkoList}, we are motivated to make the following spacing conjecture.

\begin{conjecture}[Spacing conjecture] \label{con:ss}
	The spacing constant from \eqref{eqn:c} for $\Lambda$ is positive and equal to $\lambda_6-\lambda_5$, which is greater than $159$.
\end{conjecture}

This conjecture might seem surprising, so it is worth noting that the spacing (of the real parts) of the zeros of $\Lambda$ is connected to the spacing of the zeros of the Riemann zeta function via the following equation:
\[ \Re(\lambda_{j+1}-  \lambda_{j}) = \left( \Im(\rho_{j+1})^2 -  \Im(\rho_{j})^2\right) + \left( (\tfrac{1}{2} - \Re(\rho_{j}))^2 - (\tfrac{1}{2} - \Re(\rho_{j+1}))^2\right).\]
Because of this quadratic relationship, one expects the zeros of $\Lambda$ to spread out as $j$ increases and indeed, looking at the first  two million zeros, it appears that the smallest gap happens to occur between the fifth and sixth zeros of $\Lambda.$

Heuristically, this spacing conjecture says that the gap between the real parts of two adjacent zeros of the zeta function is at least on the order of  $1/x$ where $x$ is the modulus of one of the zeros. For comparison, the average gap is something on the order of $1/\log(x).$ It is possible and even likely that isolated small gaps would give rise to Lehmer zeros or other small gaps, although we do not see an exact connection in general.  

Assuming this spacing conjecture, our results show that our hierarchy of relaxations for the Riemann hypothesis give rise to zero-free regions for the zeta function near the critical line.
For example, the $N=1$ case paired with Theorem \ref{thm:1} gives
\begin{corollary}
	If $1$-HRRH and the spacing conjecture are true
	and if some zero $\rho_k$ of the zeta function is not on the critical line, then
	\[ \left| \Im (\rho_k) \right|  \left| \tfrac{1}{2} - \Re(\rho_k)\right|  \geq \frac{159\sqrt{3}}{2\pi}.\]
\end{corollary}

This lower bound is not sharp and Remark \ref{rem:improve} gives one argument improving this estimate.
Similarly, if $N \ge 5$, Theorem \ref{thm:2} implies
\begin{corollary}
	There is a constant $M$ independent of $N$ (and $\Lambda$) such that if
	$N$-HRRH and the spacing conjecture are true
	and if some zero $\rho_k$  of the zeta function is not on the critical line, then
	\[\left|  \Im (\rho_k) \right| \left| \tfrac{1}{2} - \Re(\rho_k)\right|  \geq \frac{159 M}{2} N^{1/9}.\]
\end{corollary}

To prove these corollaries, assume some $\rho_k$ is not on the critical line and without loss of generality, assume $\Im(\rho_k) >0.$ Then $\lambda_k :=\eta_k^2$ (defined in \eqref{eqn:rhok}) is a zero of $\Lambda$ with $\Im(\lambda_k) =  2 \Im(\rho_k)(\tfrac{1}{2} - \Re(\rho_k)) \ne 0$. As $\kappa=b/c$ where $c=159$ and $b$ is defined in \eqref{eqn:b}, we have
\[ \frac{2 \Im(\rho_k)|\tfrac{1}{2} - \Re(\rho_k)|}{159}= \frac{|\Im(\lambda_k)| }{159} \ge \frac{b}{159} = \kappa,\]
and the results follow from the $\kappa$ inequalities in Theorems \ref{thm:1} and \ref{thm:2}.

These results are rather odd; zero free regions, as developed classically by de la Vall\'ee Poussin, Littlewood, Chudakov \cite{Poussin, LittlewoodZeros, Chudakov}, and in more modern terms by Cheng, Ford and other authors \cite{Cheng, KevinFord}, usually say that the zeros of the Riemann zeta function have to be away from the $s=1$ line. Instead, ours say that there are no zeros close to the critical line $s = 1/2$, see Figure \ref{fig:rh}. So, if  $N$-HRRH holds to some order $N$ but then suddenly fails, then the Riemann hypothesis has to fail badly-- the zeros need to be, in some sense, far away from the critical line. These findings appear to align with qualitative conclusions suggested by \cite{Farmer2}; indeed, Farmer's work implies that if a function satisfies a family of conditions including $1$-HRRH, then its non-real zeros should in some sense be far away from the real line.

\begin{figure}[h!]
  \centering
  \begin{subfigure}[b]{0.25\textwidth}
  \centering
    \begin{tikzpicture}[scale=.5]
      \begin{axis}[
          domain=3:5, 
          xmin=-.01, xmax=1.2,
          ymin=3, ymax=5, 
          samples=1001,
          axis y line=middle,
          axis x line=none,
          ticks=none
        ]
        \addplot[dashed,name path = C](1/2,x);
        \addplot[dashed,ultra thick,name path=S](1,x);
        \addplot[draw=red, name path=F, thick] ({  1 -  1/ (49.13* (ln(x))^(2/3) * (ln(ln(x)))^(1/3))},{x});
        \addplot[red,opacity=.4] fill between [of=F and C];
      \end{axis}
    \end{tikzpicture}
    \caption{Ford's result about the locations of allowable zeros of the zeta function. Shown here is the range from $|t|=3$ to $|t|=5$.}
  \end{subfigure}
  \qquad
  \begin{subfigure}[b]{0.25\textwidth}
  \centering
    \begin{tikzpicture}[scale=.5]
      \begin{axis}[
          domain=1:800, 
          xmin=-.01, xmax=5,
          ymin=10, ymax=499, 
          samples=1001,
          axis y line=middle,
          axis x line=none,
          ticks=none
        ]
        \addplot[dashed,name path = C](1/2,x);
        \addplot[dashed,ultra thick,name path=S](1,x);
        \addplot[white, name path = B](250,x);
        \addplot[draw=blue, name path=U, thick] ({1/2+159*sqrt(3)/(2*3.14*x)},{x});
        \addplot[blue,opacity=.4] fill between [of=B and U];
      \end{axis}
    \end{tikzpicture}
      \caption{Assuming $1$-HRRH and the spacing conjecture,  any zeros off the critical line must be in the blue region. Shown here is the range from $|t|=10$ to $|t|=500$.}
  \end{subfigure}
  \qquad
  \begin{subfigure}[b]{0.25\textwidth}
  \centering
    \begin{tikzpicture}[scale=.5]
      \begin{axis}[
          domain=70:150, 
          xmin=-.01, xmax=1.2,
          ymin=70, ymax=100, 
          samples=1001,
          axis y line=middle,
          axis x line=none,
          ticks=none
        ]
        \addplot[dashed,name path = C](1/2,x);
        \addplot[dashed,ultra thick,name path=S](1,x);
        \addplot[white, name path = B](20,x);
        \addplot[draw=blue, name path=U, thick] ({1/2+159*sqrt(3)/(2*3.14*x)},{x});
        \addplot[blue,opacity=.4] fill between [of=B and U];
        \addplot[draw=red, name path=F, thick] ({  1 -  1/ (49.13* (ln(x))^(2/3) * (ln(ln(x)))^(1/3))},{x});
        \addplot[red,opacity=.4] fill between [of=F and C];
      \end{axis}
    \end{tikzpicture}
      \caption{Our region in blue and Ford's region in red from $|t|=70$ to $|t|=100$.}
  \end{subfigure}
  \caption{Allowable regions for zeros ($\sigma+it$) of the zeta function. The heavy dashed line is $\sigma=1$ and the light dashed line is $\sigma=\frac12$.} \label{fig:rh}
\end{figure}
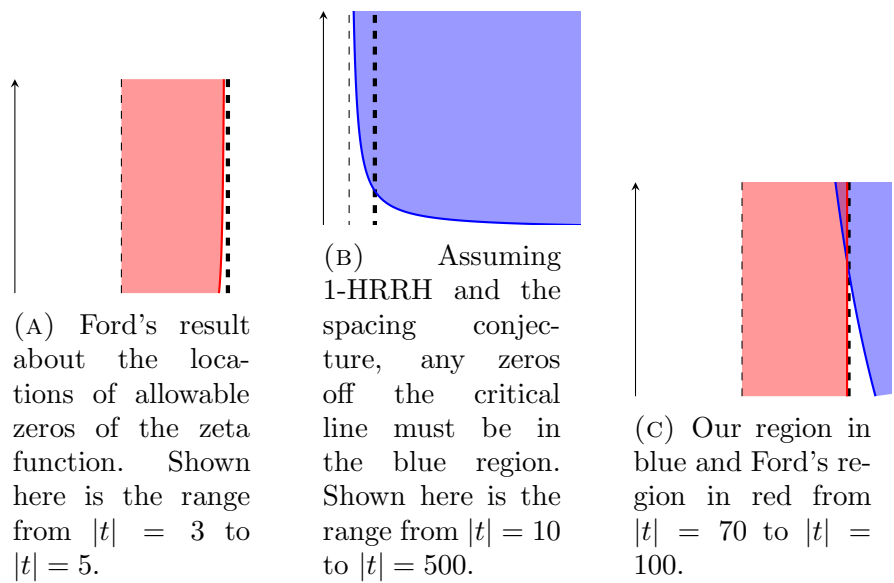


It is worth noting that our approach generally fits into the P\'olya-Jensen approach \cite{polya1,polya2} to the Riemann hypothesis, which is based upon the belief that there is some naturally occurring self-adjoint operator
with eigenvalues equal to the zeros of $\Lambda.$
Other Laguerre-P\'olya based approaches to the Riemann hypothesis include
deBranges' approach \cite{debrangesjfa}, Li's criterion \cite{LiApproach}, and
Rodgers and Tao's  establishment that the de Bruijn-Newman constant is non-negative \cite{RodgersTao}.

Related exciting results on the asymptotic hyperbolicity of certain Jensen polynomials appeared in recent work by Griffen et al.~in \cite{Ono1,Ono2}. 
 Their ideas uses power series coefficients derived from $\Lambda$ at a single point, but as discussed in the critique \cite{Farmer1}, the results do not appear to have direct implications for the Riemann hypothesis.  Our methods in some sense resolve the apparent disharmony between the Griffen et al.~and Farmer perspectives; we use power series coefficients derived from $\Lambda$ at a number of different points simultaneously and under various conjectures, are able to draw conclusions about zero-free regions.

\subsection{Alternate Spacing Regimes} \label{subsec:ssc} 
In this section, we mention two alternate (non-constant) spacing regimes for $\Lambda$, one stronger than constant spacing and one weaker than constant spacing. These could (in theory) be used in place of Conjecture \ref{con:ss} to apply our main theorems to $\Lambda$ and hence, draw conclusions about the zeros of the zeta function. 
\subsubsection{Strong Spacing} 
Under the assumption of the Riemann hypothesis, standard conjectures informed by random matrix theory \cite{Montgomery1, Odlyzko} imply that the zeros of the zeta function display quadratic repulsion. This implies the standard conjecture that $|\rho_{j+1} - \rho_j |< |\rho_j|^{-\alpha}$ at most finitely many times for $\alpha > 1/3.$ If one converts this to a statement about the zeros of $\Lambda$, it says: there is a constant $C$ (depending on $\gamma$) such that
 \begin{equation} \label{eqn:ssc1} \lambda_{j+1}-\lambda_j \geq C \lambda_j^{\gamma}\end{equation}
 for $\gamma \in[0, 1/3)$.

 If we do not assume the Riemann hypothesis (and hence, allow the $\lambda_k$ to lie off of the real line), one way to generalize  \eqref{eqn:ssc1} is the following family of conjectures:
	
		\begin{conjecture}[Strong spacing conjecture] \label{con:ssc}
		Let $\gamma \in [0, 1/3).$ There exists a $C> 0$ depending on $\gamma$ such that
			\[\Re(\lambda_{j+1}-\lambda_j) \geq C \Re(\lambda_j)^{\gamma}\]
		for all $j$ with $\lambda_j \ne \bar{\lambda}_{j+1}$ when $\Im(\lambda_j) \ne 0$.
	\end{conjecture}
	
It is not clear to us whether this spacing conjecture should be true, especially for $\gamma$ near $1/3$. Indeed, while it seems plausible that the zeros of $\Lambda$ should still repel each other in some sense, it is not clear whether that repelling would be concentrated in the real direction. Still, it seems worth investigating whether this conjecture might be supported by known paradigms, at least for $\gamma$ in some subinterval of $[0, 1/3).$
	
If Conjecture \ref{con:ssc} holds for some $\gamma >0$ and one assumes $1$-HRRH, Theorem \ref{thm:3} applies to $\Lambda$ and gives an intricate condition on any non-real zeros of $\Lambda$.  We leave it to future work (or an interested reader) to further investigate the zero constraints implied by $N$-HRRH (for $N \ge 1$) paired with Conjecture \ref{con:ssc} for different values of $\gamma$.
		
\subsubsection{Weak spacing} \label{ss:weak}

We say $\Lambda$ is \textbf{weakly spaced} if there is a $\gamma <0$ and
a $C> 0$ depending on $\gamma$ such that
			\[\Re(\lambda_{j+1}-\lambda_j) \geq C \Re(\lambda_j)^{\gamma}\]
		for all $j$ with $\lambda_j \ne \bar{\lambda}_{j+1}$ when $\Im(\lambda_j) \ne 0$.
Here, we note that one may be able to apply our results from Sections \ref{sec:spaced} and \ref{StrongSpace} to $\Lambda$ even if $\Lambda$ is only weakly spaced, via the following recipe.

	Let $\Lambda_0 = \Lambda$ and define a sequence of entire functions iteratively via the following relation
		$$\Lambda_n(z) = \Lambda_{n-1}(\sqrt{z})\Lambda_{n-1}(-\sqrt{z})$$
 and observe that the zeros of $\Lambda_n$ are given by $\lambda_j^{2^n}.$
	Suppose that for some $\gamma \in \mathbb{R}$ and all $j$ with $\lambda_j \ne \bar{\lambda}_{j+1}$ when $\Im(\lambda_j) \ne 0$, 
		\[\Re(\lambda_{j+1}-\lambda_j) \geq C \Re(\lambda_j)^{\gamma}.\]
	Then 
	\[\Re(\lambda_j)^{2n-1} \cdot \Re(\lambda_{j+1}-\lambda_j) \geq C \Re(\lambda_j)^{\gamma +2n-1}.\]
	If the real parts of the $\lambda_j$ sufficiently dominate the imaginary parts of the $\lambda_j$ (via a relation that would depend on $C$ and $n$),
	then this should yield a related constant $\hat{C}$ with
		\[\Re(\lambda_{j+1}^{2^n}-\lambda_j^{2^n}) \geq \hat{C}|\Re(\lambda^{2n}_j)|^{\frac{\gamma+2n-1}{2n}},\]
		as long as $\lambda_{j+1} \ne \bar{\lambda}_j$ when $\Im(\lambda_j) \ne 0$. 
	Then for sufficiently large $n,$ the zeros of $\Lambda_n$ would exhibit strong spacing and so, one could apply our results from Sections \ref{sec:spaced} and \ref{StrongSpace} to $\Lambda_n$.

\section{Hankel Matrices associated to Hadamard Products} \label{sec:prelim}

Before proving our main results, we need to collect some preliminary information about the Hankel matrices defined in \eqref{eqn:Nmatrix}. To that end, let $f$ be defined via Hadamard product as in \eqref{eqn:entiref} with $f(0) \ne 0$.  As mentioned earlier,  we say $f$ is in the  \textbf{$N$-th order Laguerre-P\'olya class} if $d_2 \in \mathbb{R}$ and \eqref{eqn:Nmatrix} holds for all $x \in \mathbb{R} \setminus  \{ \lambda_i\}_{i=1}^{\infty}$. Let us quickly derive the formulas for the entries of each $A_N(x)$. First, one can easily compute 
 \begin{equation} \label{eqn:g} g(z): = - \tfrac{d}{dz} \left( \log\left( f(z) \right ) \right) =- \frac{f'(z)}{f(z)} = -d_2+ \sum_{i=1}^{\infty} \frac{z}{\lambda_i(\lambda_i-z)}.\end{equation}
  For $z$ near each $x \in  \mathbb{R} \setminus \{ \lambda_i\}$, i.e. if $|z-x| < |\lambda_i-x|$, we can then write
 \[ \frac{z}{\lambda_i(\lambda_i - z)} =\frac{x}{\lambda_i(\lambda_i-x)} + \sum_{k=1}^{\infty} \left(\frac{1}{\lambda_i-x} \right)^{k+1} (z-x)^k.\]
For each  $ x\in \mathbb{R} \setminus \{ \lambda_i\}_{i=1}^\infty$, we can expand $g$  in a power series $g(z) = \sum_k a_k(x) (z-x)^k$ around any $ x$ as follows:
\[ g(z)  = -d_2 + \sum_{i=1}^{\infty} \frac{x}{\lambda_i(\lambda_i-x)} + \sum_{k=1}^{\infty} \left( \sum_{i=1}^{\infty} \left(\frac{1}{\lambda_i-x} \right)^{k+1}\right) (z-x)^k,\]
where our assumption that $\sum_i \frac{1}{|\lambda_i|^2} < \infty$ implies that the equality holds for all $z$ in some nontrivial interval centered around $x$.

The $N$-LP  property implies the following polynomial condition:

\begin{lemma} \label{lem:poly} Fix $x \in \mathbb{R} \setminus  \{ \lambda_i\}_{i=1}^{\infty}$ and $N \in \mathbb{N}$ and assume \eqref{eqn:Nmatrix} holds. Then for all $t \in \mathbb{R}$, 
\begin{equation}
	\label{eqn:q}  \sum_{i=1}^{\infty} q\left( \frac{t}{ \lambda_i-x}\right)^2 \ge 0,
\end{equation}
for all polynomials $q$ with real coefficients such that $q(0)=0$ and $\deg q \le N.$
\end{lemma}

\begin{proof} Let $q(z) = \sum_{j=1}^N q_j z^j$ with each $q_j \in \mathbb{R}.$ Fix  $\vec{d} =[ d_1 \dots d_N]^T \in \mathbb{R}^N$ with each $d_j = t^j q_j$. Then by our assumptions and the definition of $a_n(x)$, we have
\[ 0 \le \left \langle  \begin{bmatrix} a_1(x) & a_2(x) & \dots & a_{N}(x)  \\
 a_2(x) & a_3(x) & & \\
 \vdots & &\ddots &  \vdots\\
 a_{N}(x) & &\dots & a_{2N-1}(x) 
  \end{bmatrix} 
  \begin{bmatrix} d_1 \\ d_2 \\ \vdots \\ d_N \end{bmatrix},    \begin{bmatrix} d_1 \\ d_2 \\ \vdots \\ d_N \end{bmatrix}\right \rangle = \sum_{i=1}^{\infty} \left \langle A^i \vec{d}, \vec{d} \right \rangle, \]
where $A^i$ denotes the $N\times N$ matrix whose $jk^{th}$ entry is $(\frac{1}{\lambda_i-x})^{j+k}.$ Then
\[ 
\begin{aligned} \left \langle A^i \vec{d}, \vec{d} \right \rangle = \sum_{j,k=1}^N d_k d_j\left(\tfrac{1}{\lambda_i-x}\right)^{j+k} 
&= \left( \sum_{j=1}^N d_j \left(\tfrac{1}{\lambda_i-x}\right)^j \right)^2 \\
&=\left( \sum_{j=1}^N q_j \left(\tfrac{t}{\lambda_i-x}\right)^j \right)^2 = q\left( \frac{t}{ \lambda_i-x}\right)^2.
\end{aligned}\]
Combining our two equations immediately gives \eqref{eqn:q}.
\end{proof}


\section{Spaced functions}  \label{sec:spaced}
Let $f$ be a spaced function defined via a Hadamard product as in \eqref{eqn:entiref} with $f(0)\ne 0$  and let $c,b,\kappa$ denote the respective spacing constant, height, and aperture of $f$, see \eqref{eqn:c} and \eqref{eqn:b}.  Write each zero $\lambda_i = \alpha_i + i \beta_i$ for $\alpha_i, \beta_i \in \mathbb{R}$. 

The following result will be used implicitly in many of our estimates.

\begin{lemma} \label{lem:1O} Let $f$ be in the first order Laguerre-P\'olya class. If $\lambda_j$ is a zero of $f$, then  $\bar{\lambda}_j$ is also a zero of $f$. 
\end{lemma}

\begin{proof} Recall that $g(z) =- \frac{ f'(z)}{f(z)}$. Looking at the power series expansion of $f$ near each zero, it is easy to see that
\[ g'(z) = \frac{f'(z)^2-f''(z)f(z)}{f(z)^2} \] 
has poles (of order two) exactly at the zeros of $f$. As $f$ is in the first order Laguerre-P\'olya class, \eqref{eqn:Nmatrix} with $N=1$ implies that 
\[ g'(x) \in \mathbb{R} \text{ for all } x\in \mathbb{R} \setminus \{ \lambda_i\}_{i=1}^{\infty}.\]
 Thus, we can apply the Schwarz reflection principle on any interval in $\mathbb{R} \setminus \{ \lambda_i\}_{i=1}^{\infty}$ to conclude that (except at the zeros of $f$):
 \[ g'(z) = \overline{ g'(\bar{z})}.\] 
 Then this implies $g'$ has a pole at $w\in \mathbb{C}$ if only if it has a pole at $\bar{w}.$ Combining our observations gives:
 $\lambda_j$ is zero of $f$ if and only if $\lambda_j$ is a pole of $g'$ if and only if $\bar{\lambda}_j$ is a pole of $g'$ if and only if $\bar{\lambda}_j$ is a zero of $f$, which establishes the claim.
\end{proof}

\begin{theorem} \label{lem:kappa} Let $f$ be in the first order Laguerre-P\'olya class with $f(0)\ne0$. Then $\kappa  \ge \tfrac{\sqrt{3}}{\pi}.$
\end{theorem}

\begin{proof} If $\kappa =\infty$, then we are done. 
Thus, without loss of generality, assume $f$ has at least one zero in $\mathbb{C}\setminus \mathbb{R}$. Fix $\epsilon >0$ and choose $\lambda_j$ so that $0<|\beta_j |< b +\epsilon$. By reordering, assume this zero is $\lambda_1$ and its complex conjugate is $\lambda_2.$ Then setting $\alpha_1= x$ in \eqref{eqn:Nmatrix} with $N=1$ gives
\[ 0 \le \Re \big( a_1(\alpha_1) \big) \le \sum_{i\ge 3} \left| \frac{1}{\lambda_i-\alpha_1}\right|^2 -\frac{2}{|\beta_1|^2}  
\le \frac{2\pi^2}{3c^2} - \frac{2}{|b+\epsilon|^2}.\]
Rearranging and letting $\epsilon \rightarrow 0$ implies $\kappa = \tfrac{b}{c} \ge \tfrac{ \sqrt{3}}{\pi}$.
\end{proof}
 As detailed in the following remark, the estimate in Theorem \ref{lem:kappa} can be improved. 
 
\begin{remark}\label{rem:improve} Define $\lambda_1, \lambda_2$ as in the proof of Theorem \ref{lem:kappa} and consider the zeros of $f$ whose real parts are closest to $\Re(\lambda_1) = \alpha_1$. The worst case scenarios (in terms of the estimates) are either simple real zeros at $\alpha_1  \pm c$ or complex zeros in conjugate pairs at some $\pm i \beta + (\alpha_1 \pm c)$ where $b \le |\beta|.$ 
Then (handling the closest zeros to $\lambda_1,\lambda_2$ separately), our estimates become
\[ 0 \le \max \left\{ \frac{4}{c^2+b^2}, \frac{2}{c^2} \right\} +\left( \frac{2\pi^2}{3c^2} -\frac{4}{c^2} \right) - \frac{2}{|b+\epsilon|^2}.\]
Applying that refined argument to all of the zeros of $f$ (excepting $\lambda_1, \lambda_2$) yields the estimate
\[ 0 \le  \sum_{n=1}^{\infty} \max \left\{ \frac{4}{c^2n^2+b^2}, \frac{2}{c^2n^2} \right\}  - \frac{2}{|b+\epsilon|^2}.\]
Rearranging terms and letting $\epsilon \rightarrow 0$ gives
\[ 1 \le \sum_{n=1}^{\infty} \max \left \{ \frac{2\kappa^2}{\kappa^2+n^2} , \frac{\kappa^2}{n^2} \right\} \le \sum_{n=1}^{\infty} \frac{2\kappa^2}{\kappa^2+n^2},\]
where the second inequality holds if $\kappa \le 1$ and the right-hand term is increasing in $\kappa$. Thus to obtain a bound on $\kappa$, we need only solve 
\[ 1 = \sum_{n=1}^{\infty} \frac{2\kappa^2}{\kappa^2+n^2} = \kappa \pi \coth( \kappa \pi) - 1\]
for $\kappa$. This yields $\kappa \approx 0.609566$ and so, if  $f$ is in the first order Laguerre-P\'olya class with $f(0)\ne0$, then $\kappa \ge 0.60956 > 0.551329 > \tfrac{\sqrt{3}}{\pi}$, the original bound from Theorem \ref{lem:kappa}.
\end{remark} 

Now, we prove our estimate for general $N$.

\begin{theorem} \label{thm:NLP} Let $f$ be an entire function of genus one such that $f(0)\neq 0.$
		If $f$ is in the N-th order Laguerre-P\'olya class then
		$$N \leq \frac{\ln\left( \frac{4\pi^2}{3}+4 \right)}{\ln 2} + \frac{\pi^3 \sqrt{3}}	{\ln 2} \left(\kappa + \sqrt{1+\kappa^2} \right)^6\left(1+\kappa^2\right)^{3/2} \approx \kappa^9. $$
\end{theorem}

\begin{proof}  Note that the $N$-th order Laguerre-P\'olya classes are nested, namely if $f$ is in $N$-LP and $M <N$, then $f$ is also in $M$-LP. Thus, we can asssume $f$ is in $1$-LP, since otherwise the theorem statement is trivial. 

We actually prove the contrapositive of the theorem, i.e. if
\begin{equation}\label{eqn:N} N \ge \frac{\ln\left( \frac{4\pi^2}{3}+4 \right)}{\ln 2} + \frac{\pi^3 \sqrt{3}}{\ln 2} \left(\kappa + \sqrt{1+\kappa^2} \right)^6\left(1+\kappa^2\right)^{3/2} \approx \kappa^9,\end{equation}
then $f$ is not in the N-th order Laguerre-P\'olya class.

First if $\kappa = \infty$, then the result is trivially true. Thus, we can assume $\kappa <\infty$.  Set 
\[ d =\sqrt{b^2 +c^2} = c \sqrt{ 1+ \kappa^2}\] 
and $\epsilon = -b + d>0.$ Choose $\lambda_j$ such that $b \le |\beta_j|<b + \epsilon/2.$ After reordering, we can assume $j=1$ and $\lambda_2 = \bar{\lambda}_1.$ For ease of notation, define $\beta :=\beta_1$ and $\alpha:=\alpha_1$.
Set  
\begin{equation} \label{eqn:t0} \tilde{t} =  \pi^2c\left(1+\kappa^2\right)^{3/2}\left(\sqrt{1+\kappa^2}+\kappa \right)^6.
\end{equation}
Let $S_N$ denote the $N^{th}$ degree Taylor polynomial of $\sin(z)$ centered at $0$. In this proof, we will  show that if 
\begin{equation} \label{eqn:Nest} N \ge \frac{\ln\left( \frac{4\pi^2}{3} +4\right)}{\ln 2} + \frac{3\tilde{t}}{\ln(2)c} \max \left \{ 1, \tfrac{1}{\kappa}\right \} 
\end{equation}
then
\begin{equation} \label{eqn:SN} \sum_{i=1}^{\infty} \Re \left( S_N \left( \tfrac{\tilde{t}}{\lambda_i-\alpha} \right)^2 \right) <0.\end{equation}
By Lemma \ref{lem:kappa}, we know $\kappa \ge \sqrt{3}/\pi$. Using that estimate in \eqref{eqn:Nest} gives the bound in \eqref{eqn:N} and then the conclusion follows from Lemma \ref{lem:poly}, since $S_N$ has real coefficients with $S_N(0)=0$ and $\deg S_N \le N$. Thus, we just need to establish \eqref{eqn:SN}.  \\

\noindent \textbf{Step 1:} Show $\displaystyle \sum_{i=1}^{\infty} \Re\left( \sin \left( \tfrac{\tilde{t}}{\lambda_i-\alpha} \right)^2 \right) <-1.$\\

\noindent Recall that $\sin(x+iy) = \sin(x) \cosh(y) +i\cos(x) \sinh(y).$ First assume $i \ge 3$ and $t \in \mathbb{R}^+$. Then
\[ \Re \left( \sin^2\left( \frac{t}{ \lambda_i-\alpha}\right) \right) \le \sin^2\left( \Re\left( \frac{t}{ \lambda_i-\alpha}\right) \right) e^{2 \left|\Im\left( \frac{t}{ \lambda_i-\alpha}\right)\right|} \le \frac{t^2}{|\lambda_i-\alpha|^2} e^{ \tfrac{2|\beta_i| t}{|\lambda_i-\alpha|^2}}.\]
If $|\beta_i| > b+ \epsilon$, then
\[ \frac{|\beta_i|}{|\lambda_i-\alpha|^2} = \frac{|\beta_i|}{\beta_i^2 +(\alpha-\alpha_i)^2} \le \frac{1}{|\beta_i|} < \frac{1}{b +\epsilon} = \frac{1}{d}.\]
If $|\beta_i| \le b+ \epsilon$, then
\[ \frac{|\beta_i|}{|\lambda_i-\alpha|^2} \le \frac{ b +\epsilon}{b^2 + c^2} = \frac{\sqrt{b^2+c^2}}{b^2+c^2} = \frac{1}{d}.\]
This implies that 
\[ \sum_{i\ge 3} \Re \left( \sin^2\left( \frac{t}{ \lambda_i-\alpha}\right) \right) \le \sum_ {i \ge 3} \frac{t^2}{|a_i-\alpha|^2} e^{2t/d} \le \frac{2 \pi^2 t^2}{3 c^2}e^{2t/d}.\]
Similarly, if $i=1$ or $i=2$, then
\[
\begin{aligned} \Re \left( \sin^2\left( \frac{t}{ \lambda_i-\alpha}\right) \right) = -\sinh^2\left( \frac{t}{\beta_i} \right) &\le -\frac{1}{4}e^{2t/|\beta|} +\frac{1}{2} \\
&\le -\frac{1}{4}e^{\frac{2t}{b+\epsilon/2}} +\frac{1}{2} = -\frac{1}{4}e^{\frac{4t}{b+d}} +\frac{1}{2}.
\end{aligned}\]
Thus, we can conclude that
\begin{equation} \label{eqn:est1} \sum_{i=1}^{\infty} \Re \left( \sin^2\left( \frac{t}{ \lambda_i-\alpha}\right) \right) \le \frac{2\pi^2 t^2}{3 c^2}e^{\frac{2t}{d}}  -\frac{1}{2}e^{\frac{4t}{b+d}} +1.\end{equation}
Now, we just need to show that if $t$ satisfies \eqref{eqn:t0}, then $\eqref{eqn:est1} < -1.$ That inequality is equivalent to
\[ e^{\frac{2t}{d}} \left( e^{Bt} - \frac{4\pi^2}{3c^2} t^2\right) >4,\]
where 
\[ 
\begin{aligned} B = \frac{4}{b+d} -\frac{2}{d} &= 2\frac{d-b}{d(b+d)} \\
&= 2\frac{ \sqrt{1+\kappa^2} - \kappa}{c\sqrt{1+\kappa^2}(\sqrt{1+\kappa^2}+\kappa)} = \frac{2}{c\sqrt{1+\kappa^2}(\sqrt{1+\kappa^2}+\kappa)^2}.
\end{aligned}\]
Observe that $e^{\frac{2t}{d}} >4$ occurs if and only if 
\begin{equation} \label{eqn:t2}t > \frac{ d \ln(4)}{2} = \frac{ c \sqrt{1+\kappa^2} \ln(4)}{2}.\end{equation}
 Similarly, by expanding $e^{Bt}$ as a power series centered at $0$ and cancelling $1$ from both sides, we can see
 \[ e^{Bt} - \frac{4\pi^2}{3c^2} t^2 \ge 1 \]
 occurs if $\frac{B^3}{6}t^3 \ge \frac{4\pi^2}{3c^2} t^2$, or equivalently, if 
\begin{equation} \label{eqnt3} t \ge \frac{8\pi^2}{c^2 B^3} = \pi^2 c (1+\kappa^2)^{3/2} (\sqrt{1+\kappa^2} + \kappa)^6.\end{equation}
This value is larger than the value in \eqref{eqn:t2} and so we only need to choose $t$ that satisfies this inequality to guarantee $\eqref{eqn:est1} < -1.$ The formula for $\tilde{t}$ in \eqref{eqn:t0} clearly satisfies \eqref{eqnt3}, which completes the proof of Step 1.  \\

\noindent \textbf{Step 2:} Establish \eqref{eqn:SN} for the specified values of $\tilde{t}$ and $N$. \\

\noindent First, the Cauchy integral formula expanded as a power series around $0$
implies the following estimate:  if $R >0$ and  $|z|\le R$, then for any $N \in \mathbb{N}$, 
\[ |\sin(z) - S_N(z) | \le \frac{2 |z|^{N+1}}{(2R)^{N+1}} \sup_{|w| =2R} |\sin(w)| \le \frac{2 |z|^{N+1} e^{2R}}{(2R)^{N+1}}.\]
Observe that 
\[ \left | \frac{\tilde{t}}{\lambda_i -\alpha}\right| \le \max \left\{ \tfrac{\tilde{t}}{c}, \tfrac{\tilde{t}}{b} \right\} :=R>0,\]
since $\sqrt{3}/\pi \le \kappa < \infty$ implies $b,c \ne 0$. Then 
\[
\begin{aligned}
 \sum_{i\ge 3} &\left |  \sin^2 \left( \tfrac{\tilde{t}}{\lambda_i-\alpha} \right)   -   S_N^2 \left( \tfrac{\tilde{t}}{\lambda_i-\alpha} \right)\right|  \\
 &\le \sum_{i\ge 3}  \left |  \sin \left( \tfrac{\tilde{t}}{\lambda_i-\alpha} \right)  +   S_N \left( \tfrac{\tilde{t}}{\lambda_i-\alpha} \right)\right |  \left |  \sin \left( \tfrac{\tilde{t}}{\lambda_i-\alpha} \right)   -   S_N \left( \tfrac{\tilde{t}}{\lambda_i-\alpha} \right)\right | \\
 & \le \sum_{i\ge 3} 2e^R \frac{2 e^{2R}}{(2R)^{N+1}} \left | \frac{\tilde{t}}{\lambda_i-\alpha}\right|^{N+1} \\
 & \le \frac{ 4e^{3R}}{2^{N+1}} \frac{c^{N+1}}{\tilde{t}^{N+1}} \sum_{i\ge 3} \frac{\tilde{t}^{N+1}}{|\lambda_i - \alpha|^{N+1}} \le \frac{ 2 e^{3R}}{2^N} \frac{2 \pi^2}{3}. 
\end{aligned}
\]
Similarly, if $i =1,2$, we obtain
 \[
 \begin{aligned}
  \sum_{i< 3} &\left |  \sin^2 \left( \tfrac{\tilde{t}}{\lambda_i-\alpha} \right)   -   S_N^2 \left( \tfrac{\tilde{t}}{\lambda_i-\alpha} \right)\right| \\
  &= \left |  \sin^2 \left( \tfrac{\tilde{t}}{ i\beta} \right)   -   S_N^2 \left( \tfrac{\tilde{t}}{i\beta} \right)\right|  + \left |  \sin^2 \left( \tfrac{\tilde{t}}{- i\beta} \right)   -   S_N^2 \left( \tfrac{\tilde{t}}{-i\beta} \right)\right| \\
&  \le 2 \frac{2e^R\left| \tfrac{\tilde{t}}{b}\right|^{N+1} 2e^{2R}}{(2R)^{N+1}} \le  \frac{4 e^{3R}}{2^N}.
  \end{aligned}\]
Setting $\frac{ 4e^{3R}}{2^N}\left(  \frac{ \pi^2}{3} +1\right)< 1$ and solving for $N$ yields
\[N >  \frac{\ln\left( \frac{4\pi^2}{3}+4 \right)}{\ln 2}  + \frac{3R}{\ln 2}  =   \frac{\ln\left( \frac{4\pi^2}{3}+4 \right)}{\ln 2}    +\frac{3\tilde{t}}{c\ln 2} \max \left \{ 1, \tfrac{1}{\kappa}\right \},\] 
 the earlier condition on $N$. From this, we can immediately conclude that 
\[ 
\begin{aligned}
 \sum_{i=1}^{\infty}& \Re\left( S_N^2 \left( \tfrac{\tilde{t}}{\lambda_i-\alpha} \right)  \right)\\ 
 &\le \sum_{i=1}^{\infty} \Re\left(  \sin^2 \left( \tfrac{\tilde{t}}{\lambda_i-\alpha} \right) \right) + \sum_{i=1}^{\infty} \left |  \sin^2 \left( \tfrac{\tilde{t}}{\lambda_i-\alpha} \right)   -   S_N^2 \left( \tfrac{\tilde{t}}{\lambda_i-\alpha} \right)\right|  \\
 &< -1+1=0,
 \end{aligned}\]
which is what we needed to show.
\end{proof}

The proof of Theorem \ref{thm:NLP} encodes the following result, which may be of independent interest:

\begin{lemma}[Sine recovery lemma]  \label{lem:sine}
		Let $f$ be an entire function of genus one such that $f(0)\neq 0.$
		If
		$\sum_i \sin(t/(\lambda_i-x))^2 \geq 0$ for all $t \in \mathbb{R}$ and $x \in \mathbb{R} \setminus \{\lambda_i\}$, then
		$\kappa = \infty.$ 
		\end{lemma}
		
		\begin{proof} The proof of Theorem \ref{thm:NLP} assumed that $f$ was an entire function of genus one with $f(0)\neq 0$ and that $\kappa < \infty$. Given those assumptions, the proof produced numbers $\tilde{t} \in \mathbb{R}$ and $x \in  \mathbb{R} \setminus \{\lambda_i\}$ such that 
		\begin{equation} \label{eqn:sinneg} \displaystyle \sum_{i=1}^{\infty} \Re\left( \sin \left( \tfrac{\tilde{t}}{\lambda_i-x} \right)^2 \right) <-1.\end{equation}
		Since this contradicts the above summation claim, it follows  that $\kappa = \infty.$  
		\end{proof}
		
		It is worth noting that the proof of Theorem \ref{thm:NLP} also assumed that $f$ was in the first order Laguerre-P\'olya class. This gave access to the useful fact: $\kappa \ge \frac{\sqrt{3}}{\pi}$. However that assumption is not used in the part of the proof giving \eqref{eqn:sinneg} and so, it need not be an assumption of the Sine recovery lemma.
		
\section{Strongly Spaced Functions} \label{StrongSpace} Let $f$ be an entire function defined via Hadamard product as in \eqref{eqn:entiref}. Further, assume that all zeros of $f$ are simple with positive real part and any complex zeros appear in complex conjugate pairs. If $f$ is in the  first order Laguerre-P\'olya class,  Lemma \ref{lem:1O} says that the conjugate pair condition is immediate. 

Recall that $f$ is \textbf{strongly spaced} if there is some exponent $\gamma >0$ and spacing constant $C$ such that if the $\lambda_j$ are ordered with increasing real part, then
\[\Re(\lambda_{j+1}-\lambda_j) \geq C| \Re(\lambda_j)|^{\gamma}\]
for all $j$ where $\lambda_{j+1} \ne \bar{\lambda}_j$. This condition implies the following:

\begin{theorem} \label{thm:ss} Let $f$ be  in the first order Laguerre-P\'olya class with $f(0) \ne 0$ and strongly spaced with exponent $\gamma>0$ and spacing constant $C$.  If $\lambda_j$ is a zero of $f$ with $\Im(\lambda_j) \ne 0$, $\lambda_{j+1} = \bar{\lambda}_j$, and $j >1$, then 
\[ \left | \Im(\lambda_j) \right| \ge C\tfrac{\sqrt{3}}{2\pi} \left|\Re(\lambda_{\lfloor j/2 \rfloor} )\right|^{\gamma},\]
where $\lfloor j/2 \rfloor$ denotes the greatest integer less than or equal to $j/2$. Similarly, if $j=1$, then
\[  \left | \Im(\lambda_1) \right| \ge C\tfrac{\sqrt{3}}{\pi} \left|\Re(\lambda_1 )\right|^{\gamma}.\]
\end{theorem}

\begin{proof} Write each $\lambda_ i = \alpha_i + i \beta_i$. Assume $\lambda_j$ is a zero of $f$ with $\Im(\lambda_j) \ne 0$ and for now, assume $j>1$. Since $f$ is in the first order Laguerre-P\'olya class, we can substitute $x = \alpha_j =\alpha_{j+1} $  into  Lemma \ref{lem:poly} with $q(z)=z$ to conclude
\begin{equation} \label{eqn:betaj} 0 \le \sum_{i=1}^{\infty} \Re\left( \left( \frac{1}{ \lambda_i - \alpha_j}\right)^2\right) \le \sum_{i\ne j, j+1}\frac{1}{| \alpha_i - \alpha_j|^2}- \frac{2}{|\beta_j|^2}.\end{equation}
Now, write
\[ 
\begin{aligned}
\sum_{i\ne j, j+1}\frac{1}{| \alpha_i - \alpha_j|^2} &= \sum_{i=1}^{\lfloor j/2 \rfloor} \frac{1}{| \alpha_i - \alpha_j|^2} + \sum_{i=\lfloor j/2 \rfloor+1}^{j-1} \frac{1}{| \alpha_i - \alpha_j|^2}  + \sum_{i> j+1} \frac{1}{| \alpha_i - \alpha_j|^2} \\ 
& = S_1 + S_2 +S_3.
\end{aligned} \]
To handle $S_3$, observe that if $i > j+1$, then 
\[ 
\begin{aligned}
|\alpha_i - \alpha_j| &= |\alpha_i - \alpha_{i-1}| +  |\alpha_{i-1} - \alpha_{i-2}| + \dots + |\alpha_{j+2} - \alpha_{j+1}| \\
&\ge \tfrac{(i-(j+1))}{2} C | \alpha_{j+1}|^{\gamma},
\end{aligned}
\]
where we used the fact that at most $1/2$ of the terms in the sum can be zero and the rest are bounded below by $C  | \alpha_{j+1}|^{\gamma}$. Then 
\begin{equation} \label{eqn:S3} S_3 \le  \sum_{i> j+1} \frac{4}{C^2 (i-(j+1))^2 | \alpha_{j+1}|^{2\gamma}}  = \frac{2 \pi^2}{3 C^2} \frac{1}{|\alpha_j|^{2\gamma}} \le \frac{2 \pi^2}{3 C^2} \frac{1}{|\alpha_{\lfloor j/2 \rfloor}|^{2\gamma}} . \end{equation}
Similarly, for $S_2$, we have if $ i \ge \lfloor j/2 \rfloor$, then 
\[
\begin{aligned}  |\alpha_j - \alpha_i| &= |\alpha_j - \alpha_{j-1}| +  |\alpha_{j-1} - \alpha_{j-2}| + \dots + |\alpha_{i+1} - \alpha_{i}| \\
&\ge \tfrac{j-i}{2} C | \alpha_{i}|^{\gamma}, \\ 
&\ge \tfrac{j-i}{2} C | \alpha_{\lfloor j/2 \rfloor}|^{\gamma},
\end{aligned} 
\]
and so,
\[S_2 \le \sum_{i=\lfloor j/2 \rfloor+1}^{j-1} \frac{4}{C^2 (j-i)^2 |  \alpha_{\lfloor j/2 \rfloor}|^{2\gamma} } \le \frac{2 \pi^2}{3 C^2} \frac{1}{|\alpha_{\lfloor j/2 \rfloor}|^{2\gamma}}.\]
Lastly, for $S_1$ observe that if $i \le {\lfloor j/2 \rfloor}$, then
\[ 
\begin{aligned}
|\alpha_j-\alpha_i| &\ge | \alpha_j - \alpha_{\lfloor j/2 \rfloor} | \ge \tfrac{(j- \lfloor j/2 \rfloor)}{2}C  |\alpha_{\lfloor j/2 \rfloor}|^{ \gamma} \ge \tfrac{j}{4} C |\alpha_{\lfloor j/2 \rfloor}|^{ \gamma},
\end{aligned} \]
and so,
\[ S_1 \le \sum_{i=1}^{\lfloor j/2 \rfloor} \frac{16}{j^2 C^2 |\alpha_{\lfloor j/2 \rfloor}|^{2 \gamma}} \le \frac{8}{j C^2  |\alpha_{\lfloor j/2 \rfloor}|^{2 \gamma}}.\]
Combining our estimates for $S_1, S_2, S_3$ with \eqref{eqn:betaj} give 
\[ 0 \le \frac{8}{j C^2  |\alpha_{\lfloor j/2 \rfloor}|^{2 \gamma}} +  \frac{4 \pi^2}{3 C^2} \frac{1}{|\alpha_{\lfloor j/2 \rfloor}|^{2\gamma}} -  \frac{2}{|\beta_j|^2},\]
or equivalently
\[ |\beta_j| \ge C {|\alpha_{\lfloor j/2 \rfloor}|^{\gamma}} \frac{1}{\sqrt{ \tfrac{4}{j} + \tfrac{2\pi^2}{3}}} \ge C {|\alpha_{\lfloor j/2 \rfloor}|^{\gamma}} \frac{1}{\sqrt{\tfrac{2\pi^2}{3} + \tfrac{2\pi^2}{3}}} ,\]
which implies the desired inequality. 

Lastly, if $j=1$, then both $S_1$ and $S_2$ are trivial and the second-to-last inequality in \eqref{eqn:S3} gives
\[ S_3 \le  \frac{2 \pi^2}{3 C^2} \frac{1}{|\alpha_1|^{2\gamma}}.\]
 and substituting this into \eqref{eqn:betaj} gives the desired inequality.
\end{proof}

\end{document}